\documentclass[11pt,letterpaper]{amsart}
\usepackage{amscd}
\usepackage{amssymb}
\usepackage{amsthm}
\usepackage{enumerate}
\usepackage[normalem]{ulem}
\usepackage{mathtools}
\usepackage[usenames,dvipsnames]{xcolor}
\usepackage{stmaryrd}
\usepackage[left=1in,top=1in,right=1in]{geometry}
\usepackage[T2A]{fontenc}
\usepackage[utf8]{inputenc}
\usepackage{mathrsfs}

\newtheorem{thm}{Theorem}[section]
\newtheorem{lemm}[thm]{Lemma}
\newtheorem{prop}[thm]{Proposition}
\newtheorem{cor}[thm]{Corollary}

\newtheorem{conj}{Conjecture}

\theoremstyle{definition}
\newtheorem{defn}[thm]{Definition}

\newtheorem{rem}[thm]{Remark}

\numberwithin{equation}{section}

\newcommand{\bC}{{\mathbb C}}

\newcommand{\bN}{{\mathbb N}}

\newcommand{\bZ}{{\mathbb Z}}

\newcommand{\cH}{{\mathcal H}}

\newcommand{\cL}{{\mathcal L}}
\newcommand{\cM}{{\mathcal M}}

\newcommand{\cS}{{\mathcal S}}

\DeclareMathOperator{\Hamm}{Hamm}

\DeclareMathOperator{\tr}{tr}
\DeclareMathOperator{\Aut}{Aut}

\newcommand{\ip}[1]{\langle #1 \rangle}

\newcommand\Sym{\operatorname{Sym}}

\begin{document}



\title{On sofic approximations of non amenable groups}

\author{Ben Hayes}
\address{\parbox{\linewidth}{Department of Mathematics, University of Virginia, \\
141 Cabell Drive, Kerchof Hall,
P.O. Box 400137
Charlottesville, VA 22904}}
\email{brh5c@virginia.edu}
\urladdr{https://sites.google.com/site/benhayeshomepage/home}

\author{Srivatsav Kunnawalkam Elayavalli}
\address{\parbox{\linewidth}{Department of Mathematics, University of California, \\
San Diego, 9500 Gilman Drive \# 0112, La Jolla, CA 92093}}
\email{srivatsav.kunnawalkam.elayavalli@vanderbilt.edu}
\urladdr{https://sites.google.com/view/srivatsavke}

\thanks{B. Hayes gratefully acknowledges support from the NSF grant DMS-2144739. S. Kunnawalkam Elayavalli gratefully acknowledges support from the Simons Postdoctoral Fellowship.}

\begin{abstract}
In this paper  we exhibit for every non amenable group that is initially sub-amenable (sometimes also referred to as LEA), two sofic approximations that are not conjugate by any automorphism of the universal sofic group. This  addresses a question of P\v{a}unescu and generalizes the Elek-Szabo uniqueness theorem for sofic approximations.
\end{abstract}

\maketitle


\section{Introduction}

In group theory and functional analysis, one is frequently interested in the study of approximate homomorphisms of groups. These are sequences of maps $\phi_{n}\colon G\to H_{n}$ where $G,H_{n}$ groups, each $H_{n}$ has a bi-invariant metric $d_{n}$ and the sequence satisfies
\[\lim_{n\to\infty}d_{n}(\phi_{n}(gh),\phi_{n}(g)\phi_{n}(h))=0.\]
Part of the interest in this stems from the situation where $H_{n}$ is in some sense ``finitary" (e.g. $H_{n}$ could be finite, or it could be a subgroup of the invertible operators on a finite-dimensional space). Relaxing the homomorphism assumptions facilitated generalizations of many interesting techniques to broader classes of groups. Particular instances of this philosophy have led to Gromov's notion of soficity of groups, Connes embeddability of groups (sometimes called hyperlinearity) and their group von Neumann algebras (see  \cite{surveysofic} for a survey), linear soficity (see \cite{linearsofic}), weak soficity (see \cite{glebskyrivera}) etc. These notions have led to the resolution of several conjectures, as well as created entirely new fields of investigation, such as sofic entropy (see \cite{surveybowen}).

While approximate homomorphisms have great utility, an inherent difficulty in their study is that, while we have many examples of groups having interesting approximate homomorphisms, the classification of approximate homomorphisms themselves is incredibly challenging. A natural candidate for, say, sofic approximations is uniqueness up to \emph{asymptotic conjugacy}: namely given a pair $\sigma_{n},\phi_{n}\colon G\to \Sym(d_{n})$ classifying when there is a $p_{n}\in \Sym(d_{n})$ so that $p_{n}\sigma_{n}(g)p_{n}^{-1}\approx \phi_{n}(g)$ as $n\to\infty$ in the Hamming distance. A celebrated result, which is implicitly proved by Kerr-Li  \cite[Lemma 4.5]{KLi2}, and explicitly stated and proved independently by Elek-Szabo \cite{elekxzabo}, states that when $G$ is amenable any two sofic approximations are asymptotically conjugate (the analogous statement for amenable groups in the context of Connes embedding was first proved by Connes \cite{Connes}). Separate proofs were later given in \cite[Theorem 5.1 and Corollary 5.2]{PoppArg} and \cite[Lemma 4]{JHel}
.
The converse, that a non amenable sofic group admits two non conjugate embeddings  is also strikingly established in \cite{elekxzabo}, which is an analogue of the celebrated Jung's theorem \cite{JungTubularity} (see also \cite{ScottSri2019ultraproduct, AGKE-GeneralizedJung} for generalizations of Jung's theorem that also  motivate the present work).

In \cite{Paunescu} P\v{a}unescu, following up on the operator algebraic setting in Brown \cite{Browntopological}, investigates the space of approximate homomorphisms  of a group $G$ modulo approximate conjugacy and shows that it has a natural convex structure. The Kerr-Li, Elek-Szabo uniqueness theorems tells us this space is a point when $G$ is amenable, and P\v{a}unescu showed that it is a nonseparable metrizable space when $G$ is not amenable.

A clean way to phrase this is by ultraproduct techniques. By taking a suitable metric ultraproduct of symmetric groups, a sofic approximation turns into a genuine injective homomorphism and approximate conjugacy turns into genuine conjugacy. We call such an ultraproduct a \emph{universal sofic group} (see Section \ref{sec:proof of main} for the precise definition).
Our  main result addresses a question of P\v{a}unescu (see  the second paragraph in the introduction of \cite{paunescuaut}) which asks about replacing conjugation with arbitrary automorphisms in the Elek-Szabo Theorem (Theorem 2 of \cite{elekxzabo}). We say that two sofic  embedddings $\pi_{1},\pi_{2}$ of $G$ into a universal sofic group $\cS$ are \emph{automorphically conjugate} if there is an automorphism $\Phi\in \Aut(\cS)$ so that $\Phi\circ\pi_{1}=\pi_{2}$. 
We say that $G$ satisfies the \emph{generalized Elek-Szabo property} if  any two sofic embeddings $\pi_1,\pi_2$ of $G$ into a universal sofic group $\mathcal{S}$ are automorphically conjugate.
The following is an if and only if characterization of the above property, in the presence of one additional assumption.

\begin{thm}\label{Gen Elek Szabo}
   Let $G$ be a sofic group that is initially subamenable. Then $G$ satisfies the generalized Elek-Szabo property if and only if $G$ is amenable.
\end{thm}

Unpublished work of Farah-Hart-Sherman shows that the continuum hypothesis implies that there are outer automorphisms of the universal sofic group (see Corollary 2.6.7 of \cite{CapLup}). Thus one may a priori ask whether or not a group satsifies the generalized Elek-Szabo property is independent of ZFC. Our work shows an unconditional result for initially subamenable groups. 

A group is \emph{initially subamenable} if for all $F\subseteq G$ finite, there is an amenable group $H$ and an injective map $\phi\colon F\to H$ so that $\phi(xy)=\phi(x)\phi(y)$ for all pairs $(x,y)\in F\times F$ with the property that $xy\in F$ (i.e. $\phi$ is a homomorphism ``when it makes sense to be").
Equivalently, a group $G$ is initially subamenable if it is the limit of a sequence of amenable groups in the space of marked groups. In \cite{CSMC}, this property is referred to as LEA or locally embeddable into amenable groups.    This is a large family of sofic groups containing all residually finite groups.  Gromov asked if all sofic groups are initially subamenable (see \cite{Gromov1}) and this was answered in the negative by Cornulier in \cite{cornulier}.

\begin{rem}
 One of our motivations for obtaining Theorem \ref{Gen Elek Szabo} is the existence problem of a non sofic group. Recall that an amalgamated product of sofic groups with amenable amalgam is sofic (see \cite{elekxzabo}). In the absence of the amenability assumption for the amalgam, the proof breaks down early on because of the inability to  ``patch'' the embedding in the amalgam. Therefore, the first step to constructing a non sofic group in this way, is to  carefully understand the space of sofic embeddings up to conjugacy.
\end{rem}


The proof of the above theorem combines three ideas: first is the identification of $L^{\infty}$ functions on the Loeb measure space inside the weak operator topology closure of the span of the universal sofic group viewed inside a matrix ultraproduct. We also have to show that any automorphism of the universal sofic group induces an automorphism of $L^{\infty}$ functions on the Loeb measure space which intertwines the actions of the universal sofic group on the Loeb measure space. Both of these rely on a theorem of P\v{a}unescu that every automorphism of the universal sofic group is pointwise inner \cite{paunescuaut}.
The second idea is the theory of extreme points of sofic embeddings of non amenable groups in the work of P\v{a}unescu \cite{Pau1}. The third idea is the
original result of Elek-Szabo (Theorem 2 in \cite{elekxzabo}) which allows one to construct an embedding whose commutant does not act ergodically on the Loeb measure space. The proof is carried out in Section 3.

\begin{rem}
 Our proof of Theorem \ref{Gen Elek Szabo} essentially shows that if $G$ is a non amenable group that admits a sofic embedding whose commutant action on the Loeb space is ergodic, then $G$ does not satisfy the generalized Elek-Szabo property.  Hence the initial subamenability assumption in Theorem \ref{Gen Elek Szabo} can be removed if there is a positive answer to the following conjecture.
\end{rem}

\begin{conj}
    Every sofic group $G$ admits an embedding $\pi: G\to \mathcal{S}$  whose centralizer acts ergodically on the Loeb measure space.
\end{conj}

We remark that in \cite[Theorem 4.8]{Paunescu} that if one works with the ultrapower of the full group of the hyperfinite equivalence relation instead of the universal sofic group, then having a centralizer which acts erogdically characterizes being an extreme point in the space of sofic embeddings. Thus one should expect the above conjecture to follow by proving that the convex structure in \cite{Pau1} always has an extreme point.

\begin{rem}
    We thank the anonymous referee for pointing out that our results also apply in the situation of multiples: one can consider various universal sofic groups by considering instead $\mathcal{S}_f:=\prod_{n\to \omega}(\text{Sym}_{nf(n)}, d_{nf(n)})$ where $f:\mathbb{N}\to \mathbb{N}$, and naturally consider for any sofic embedding $\pi: G\to \mathcal{S}$ the amplification $\pi_f:G\to \mathcal{S}_f$ given by $\pi_f(g)= (\pi(g)\otimes 1_{f(n)})_{n\to\omega}$. From Proposition 2.8 of \cite{Pau1} and the fact that Elek-Szabo's original proof handles the multiples case as well (see also Theorem 6.18 in \cite{ScottSri2019ultraproduct}), we arrive at the result that if $G$ is a non amenable initially subamenable group then  $G$ admits two sofic embeddings $\pi_{1},\pi_{2}$ so that all their multiples are not automorphically conjugate (see Section \ref{mul} and in particular,
    Corollary \ref{cor: multiples, I guess} for a precise statement).
\end{rem}

\section{Notation and Preliminaries}


\subsection*{Groups with bi-invariant metrics:} Throughout we consider $G$ to be a countable group. Let $\text{Sym}_n$ denote the finite symmetric group of rank $n$. Recall the normalized  Hamming distance which is a bi-invariant metric on $\text{Sym}_n$: $$d_n(\sigma, \rho)= \frac{|\{i\ |\ \sigma(i)\neq \rho(i)\}|}{n}.$$ Recall the following:

\begin{defn}
    A sequence of maps $\sigma_{n}\colon G\to \Sym_n$ is said to be an \emph{approximate homomorphism} if for all $g,h\in G$ we have $$\lim_{n\to \infty} d_n(\sigma_n(gh), \sigma_n(g)\sigma_n(h))=0.$$
\end{defn}

\begin{defn}
    A sequence of maps $\sigma_{n}\colon G\to \Sym_n $ is said to be a \emph{sofic approximation} if  $(\sigma_n)_{n}$ is an approximate homomorphism  and for
    all $g\in G\setminus\{e\}$ we have $$\lim_{n\to \infty} d_n(\sigma_n(g), 1_n)=1.$$
\end{defn}

For the results in Section 3 we will need the notion of ultraproducts of groups with bi-invariant metrics. Let $\omega$ be a free ultrafilter on $\mathbb{N}$. Let $(G_n,d_n)$ be countable groups with bi-invariant metrics. Denote by $$\prod_{n\to \omega} (G_n,d_n)= \frac{\{(g_n)_{n\in \mathbb{N}}\}}{\{(g_n)| \lim_{n\to \omega} d_n(g_n,1_{G_n})=0\}}. $$ Observe that, by the bi-invariance property of the metrics $d_n$, the subgroup   $\{(g_n)| \lim_{n\to \omega} d_n(g_n,1_n)=0\}$ is indeed a normal subgroup.

\subsection*{Tracial von Neumann algebras:} Let $\cH$ be a Hilbert space. Recall that a unital $*$-subalgebra $M$ of $\mathbb{B}(\mathcal{H})$ is said to be a \emph{von Neumann algebra} if it is closed in the weak operator topology given by the convergence  $T_n\to T$ if $\langle (T_n-T)\xi, \eta\rangle\to 0$ for all $\xi,\eta\in \mathcal{H}$. A  \emph{normal homomorphism} between von Neumann algebras $M,N$ is a linear $\Theta\colon M\to N$ which preserves products and adjoints and such that $\Theta\big|_{\{x\in M:\|x\|\leq 1\}}$ is weak operator topology continuous. Such maps are automatically norm continuous \cite[Proposition 1.7 (e)]{ConwayOT}.
We say that $\Theta$ is an \emph{isomorphism} if it is bijective, it is then automatic that $\Theta^{-1}$ is a normal homomorphism \cite[Proposition 46.6]{ConwayOT}.
A pair $(M,\tau)$ is a tracial von Neumann if $M$ is a von Neumann algebra, and $\tau$ is a \emph{trace}, meaning that $\tau: M\to \bC$,  is a  positive linear functional such that $\tau(ab)=\tau(ba)$ and $\tau(1)=1$, and so that $\tau\big|_{\{x\in M:\|x\|\leq \}}$ is weak operator topology continuous.
Given a Hilbert space $\cH$ and $E\subseteq B(\cH)$, we let $W^{*}(E)$ be the von Neumann algebra generated by $E$.

We need the following folklore result (see the discussion in \cite[Section 2]{BekkaOAsuperrigid} for a proof).
\begin{cor}\label{lemm:Characters give conjugacy}
Let $G$ be a group,  and $(M_{j},\tau_{j}),j=1,2$ tracial von Neumann algebras. Suppose that for $j=1,2,$ we have representations $\pi_{j}\colon G\to \mathcal{U}(M_{j})$. If $\tau_{1}\circ \pi_{1}=\tau_{2}\circ \pi_{2}$, then there is a unique normal isomorphism $\Theta\colon W^{*}(\pi_{1}(G))\to W^{*}(\pi_{2}(G))$ with $\Theta\circ \pi_{1}=\pi_{2}$.
\end{cor}

We also need the notion of ultraproducts of tracial von Neumann algebras: Let $\omega$ be a free ultrafilter on $\mathbb{N}$. Let $(N_k, \tau_k)$ denote a sequence of tracial von Neumann algebras.  Denote the ultraproduct by $$\prod_{k\to \omega} (N_k,\tau_k)= \{(x_k)_{k\in \mathbb{N}}\ |\ \sup_{k} \|x_k\|< \infty\}/\{(x_k)| \lim_{k\to \omega} \|x_k\|_2=0\}. $$
If $(x_{k})_{k}\in \prod_{k}N_{k}$ with
 $\sup_{k}\|x_{k}\|<+\infty,$ we use $(x_{k})_{k\to\omega}$ for its image in $\prod_{k\to\omega}(N_{k},\tau_{k})$.
By the proof of \cite[Lemma A.9]{BrownOzawa2008}
the ultraproduct is a tracial von Neumann algebra and is equipped with a canonical trace $\tau((x_n)_{n\to \omega})= \lim_{n\to \omega}\tau_n(x_n)$.

\section{Proof of the main result}\label{sec:proof of main}
Let $\omega\in \beta(\mathbb{N})\setminus \mathbb{N}$ be a non principal ultrafilter on $\mathbb{N}$. Denote by $\mathcal{S}:=\prod_{n\to \omega}(\text{Sym}_{n}, d_n)$, where $\text{Sym}_n$ denote the symmetric groups on $n$ points. We call $\cS$ a \emph{universal sofic group}. Denote by $\chi$ the trace on $\mathcal{S}$, given by $\chi((p_n)_{n\to\omega})= 1-\lim_{n\to \omega} d_{n}(1,p_n)$ where  we recall that $d_n$ is the normalized Hamming metric on $\text{Sym}$. For an $n\in \bN$ define $\tr\colon M_{n}(\bC)\to \bC$ by
\[\tr(A)=\frac{1}{n}\sum_{j=1}^{n}A_{jj}.\]
Set $\cM:=\prod_{n\to \omega}(\mathbb{M}_n(\mathbb{C}),\tr)$, and let $\tau_{\mathcal{M}}$ be the trace on $\mathcal{M}$. By identifying each permutation with the corresponding permutation matrix we get an embedding $\mathcal{S}\leq \mathcal{U}(\prod_{n\to \omega}\mathbb{M}_n(\mathbb{C})):=\mathcal{M}$.  We let $L^{2}(\mathcal{M}),L^{2}(W^{*}(\mathcal{S}))$ be the Hilbert space completions of $\mathcal{M}, W^{*}(\mathcal{S})$ under the inner product
\[\ip{x,y}=\tau_{\mathcal{M}}(y^{*}x)\]
the inclusion $W^{*}(\mathcal{S})\hookrightarrow \mathcal{M}$ naturally extends to an isometry $L^{2}(W^{*}(\mathcal{S}))\hookrightarrow L^{2}(\mathcal{M})$ and because of this we will identify $L^{2}(W^{*}(\mathcal{S}))$ with a subspace of  $L^{2}(\mathcal{M})$. Observe that $\chi=\tau_{\mathcal{M}}|_{\mathcal{S}}$.

We will need the following fact that shows that automorphisms of the universal sofic group $\mathcal{S}$ extend nicely to the von Neumann algebra $W^*(\mathcal{S})$ generated by $\mathcal{S}$.

\begin{lemm}\label{lemm: extending to an aut of the vNa}
For any $\Phi\in Aut(\cS)$, there exists a unique trace preserving $*$-automorphism $\Phi_{*}:W^*(\cS)\to W^*(\cS)$ such that $\Phi_{*}|_{\cS}= \Phi$, where $W^*(\mathcal{S})$ denotes the von Neumann algebra generated by $\mathcal{S}$.
\end{lemm}

\begin{proof}
From the main result of \cite{paunescuaut} for any element $x\in \mathcal{S}$, there exists $y\in \cS$ satisfying $\Phi(x)=yxy^{-1}$, therefore $\chi(x)=\chi(\Phi(x))$. The conclusion then follows from Corollary \ref{lemm:Characters give conjugacy}.

\end{proof}

Denote by $L^{\infty}(\mathcal{L})$ the Loeb algebra, that is the ultraproduct of the diagonal subalgebras $\prod_{n\to \omega}D_n(\mathbb{C})$ in $\mathcal{M}$. The notation here is justified because there is a point realization of $\prod_{n\to \omega}D_n(\mathbb{C})$. Namely, there is an underlying probability measure space $(\cL,\mu)$ so that $\prod_{n\to \omega}D_n(\mathbb{C})$ is naturally isomorphic to $L^{\infty}$ functions on $(\cL,\mu)$. The actual space is inconsequential for our results; we just need the algebra of functions on the space. For this reason, we do not give the  construction of the space and instead refer the reader to \cite{Loebmeasure} for its definition. We remark that we may view $\cS$ itself in von Neumann algebraic terms. In fact, by arguments similar to \cite[Remark 3.4]{PopaMaxAbelian} one can show that $\cS$ is the normalizer of $L^{\infty}(\cL)$ inside $\cM$ (e.g. see the discussion following Theorem 0.3 in \cite{PoppArg}). In this framework, our consideration of $W^{*}(\cS)$ is similar to the notion of approximate conjugacy in \cite{APApproxEquiv}.

The following lemma is the main  place  where we use von Neumann algebra  closures in the proof of our first main result.  

\begin{lemm}\label{Loeb measure space is in the von Neumann algebra generated by S}
For any projection $p\in L^\infty(\mathcal{L})$ there exists an element $\psi\in \cS$ such that $\psi^n\to p$ in the weak operator topology. In particular, $L^\infty(\mathcal{L})\subset W^*(\mathcal{S})$.

\end{lemm}

\begin{proof}
By \cite[Lemma 5.4.2 (i)]{AP} we may find a sequence $E_{n}\subseteq \{1,\cdots,n\}$ so that $p=(1_{E_n})_{n\to\omega}$. Note that if $\lim_{n\to\omega}\frac{|E_{n}^{c}|}{n}=0$, then $p=1$, and we can simply chose $\psi=1$. So we may, and will, assume that
\[\lim_{n\to\omega}\frac{|E_{n}^{c}|}{n}>0\]
and thus that $p\ne 1$.
Choose a sofic approximation $\sigma_{n}\colon \bZ\to \Sym(E_{n}^{c}).$ E.g. we can construct $\sigma_{n}$ by letting $\phi\in \Sym(E_{n}^{c})$ have order $|E_{n}^{c}|$ and setting $\sigma_{n}(k)=\phi^{k}$.
Define $\hat{\psi}=(\sigma_{n}(1))_{n\to\omega} \in  \prod_{n\to \omega}\Sym(E_{n}^{c})$.
Regard $\hat{\psi}\in \mathcal{U}((1-p)\cM(1-p))$, and equip $(1-p)\cM(1-p)$ with the trace $\tau_{1-p}(x)=\frac{1}{\tau(1-p)}\tau(x)$
(this trace is well defined as $1-p\ne 0$).
Then, for every $k\in \bZ$:
\[\tau_{1-p}(\widehat{\psi}^{k})=\lim_{n\to\omega}\frac{1}{|E_{n}^{c}|}|\{j\in E_{n}^{c}:\sigma_{n}(k)(j)\ne j\}|=\lim_{n\to\omega}d_{\Hamm}(\sigma_{n}(1)^{k},1)=\delta_{k=0}=\ip{\lambda(k)\delta_{0},\delta_{0}},\]
where $\lambda\colon \bZ\to \mathcal{U}(\ell^{2}(\bZ))$ is the left regular representation.
By Corollary \ref{lemm:Characters give conjugacy},
this implies that we have an injective, normal $*$-homomorphism
\[\Theta\colon W^{*}(\lambda(\bZ))\to \mathcal{U}((1-p)\cM(1-p))\]
satisfying $\Theta(\lambda(k))=\widehat{\psi}^{k}$. Since normal $*$-homomorphisms are WOT continuous on the unit ball,
\[WOT-\lim_{k\to\infty}\widehat{\psi}^{k}=\Theta(WOT-\lim_{k\to\infty}\lambda(k))=0.\]
Set $\psi_n=\text{id}_{E_n}\bigsqcup \sigma_{n}(1)$,
and $\psi=(\psi_n)_{n\to\omega}=p+\widehat{\psi}$. Note that $\widehat{\psi}\in \mathcal{U}((1-p)\cM(1-p))$ and thus $\widehat{\psi}p=0=\widehat{\psi}p$.
We now compute:
\[WOT-\lim_{n\to\infty}\psi^{k}=WOT-\lim_{n\to\infty}p+\widehat{\psi}^{k}=p,\]
as required.
Being a von Neumann algebra, $L^{\infty}(\cL)$ is the norm closure of the linear span of its projections \cite[Proposition 13.3 (i)]{ConwayOT}, proving the ``in particular" part.
\end{proof}

We  have moreover that any automorphism of $\mathcal{S}$ preserves $L^{\infty}$-functions on the Loeb measure space.

\begin{lemm}\label{auts fix the Loeb measure space}
For any $\Phi\in Aut(\mathcal{S})$ we have that $\Phi_{*}(L^\infty(\mathcal{L}))= L^\infty(\mathcal{L})$. Moreover, $\Phi_{*}$ preserves the trace on $L^{\infty}(\mathcal{L})$.
\end{lemm}
\begin{proof}
The fact that $\Phi_{*}$ preserves the trace on $L^{\infty}(\cL)$ is contained in Lemma \ref{lemm: extending to an aut of the vNa}. So we focus on proving that $\Phi_{*}(L^{\infty}(\cL))\subseteq L^{\infty}(\cL)$.
Fix a projection $p\in L^\infty(\mathcal{L})$. From Lemma \ref{Loeb measure space is in the von Neumann algebra generated by S}, choose $\psi\in \mathcal{S}$ such that $\psi^n\to p$ weakly. Since $\Phi$ is pointwise inner by \cite{paunescuaut}, we see that $\Phi(\psi)=\phi\psi \phi^*$ for some $\phi\in \mathcal{S}$. Now we compute $\Phi_{*}(p)= \Phi_{*}(\lim_{n\to \infty}^{WOT}\psi^n) = \lim_{n\to \infty}^{WOT}\Phi(\psi^n)= \lim_{n\to \infty}^{WOT}\Phi(\psi)^n= \lim_{n\to \infty}^{WOT}\phi\psi^n\phi^*= \phi(\lim_{n\to \infty}^{WOT}\psi^n)\phi^*= \phi p \phi^*\in L^\infty(\mathcal{L})$ as required.
Since $L^{\infty}(\mathcal{L})$ is the norm closed linear span of its projections \cite[Proposition 13.3]{ConwayOT}, we have that $\Phi_{*}(L^{\infty}(\mathcal{L}))\subseteq L^{\infty}(\mathcal{L})$. Repeating the argument with $\Phi$ replaced by $\Phi^{-1}$ proves the opposite inclusion.
\end{proof}

Note that $\cS$ acts naturally on $L^{\infty}(\cL)$ by
\[\phi\cdot f=\phi f\phi^{-1}.\]
Observe that if $\phi=(\phi_{n})_{n\to\omega}\in \cS$ and $f=(f_{n})_{n\to\omega}\in L^{\infty}(\cL)$, then $\phi\cdot f=(\phi_{n} f_{n}\phi_{n}^{-1})_{n\to\omega}$. If we identify a diagonal matrix in $D\in M_{n}(\bC)$ with an element of $f\in \ell^{\infty}(n)$, then $\phi_{n} D \phi_{n}^{-1}$ corresponds to $f\circ \phi_{n}^{-1}$. So we can think of this action as being induced from the natural action of $\Sym_{n}$ on $\{1,\cdots,n\}$.
The last piece of the argument involves identifying a useful invariant of automorphic equivalence of embeddings, 
for which we need some terminology.
If $H$ is a subgroup of $\cS$, then we say that $H$ \emph{acts ergodically on $L^{\infty}(\cL)$} if
\[\{f\in L^{\infty}(\cL)\ |\ hfh^{-1}=f \textnormal{ for all $h\in H$}\}=\bC1.\]
Given a subgroup $H$ of $\cS$ we let
\[H'\cap \cS=\{\phi\in \cS\ |\ \phi h=h\phi \textnormal{ for all $h\in H$}\}.\]
We say that a sofic approximation $\sigma_{n}\colon G\to \Sym_{n}$ has \emph{ergodic commutant with respect to $\omega$} if $\sigma(G)'\cap \cS$ acts ergodically on $L^{\infty}(\cL)$, where $\sigma(g)=(\sigma_{n}(g))_{n\to\omega}$. We will often drop the ``with respect to $\omega$" if it is clear from context.

\begin{lemm}\label{main lemma Jung}
Let $\pi_1, \pi_2: G\to \mathcal{S}$ be two sofic embeddings of $G$ such that there exists $\Phi\in Aut(\cS)$ satisfying
$\pi_1= \Phi\circ\pi_2$. Then $\pi_1(G)'\cap \cS$ acts ergodically on $L^\infty(\mathcal{L})$ if and only if $\pi_2(G)'\cap \cS$ acts ergodically on $L^\infty(\mathcal{L})$.
\end{lemm}
\begin{proof}
Assume that $\pi_1(G)'\cap \cS$ acts ergodically on $L^\infty(\mathcal{L})$. Suppose $f\in L^\infty(\mathcal{L})$ is  such that $\phi f \phi^* = f$ for all $\phi\in \pi_2(G)'\cap \cS$. Then we have $\Phi(\phi) \Phi_{*}(f) \Phi(\phi)^* = \Phi_{*}(f)$. By Lemma \ref{auts fix the Loeb measure space} we have that $\Phi_{*}(f)\in L^{\infty}(\mathcal{L})$. By ergodicity of the action of  $\pi_1(G)'\cap \cS$, this means that $\Phi_{*}(f)= \lambda 1$ for some constant $\lambda$. Since $\Phi_{*}$ is a von Neumann algebra automorphism, this implies that $f=\lambda 1$.
The reverse implication follows by replacing $\Phi$ with $\Phi^{-1}$.
\end{proof}

We collect two more results due to P\v{a}unescu (adapting work of Kerr-Li) that are crucial.

\begin{lemm}[Theorem 2.13 in \cite{Pau1}, following up on Theorem 5.8 of \cite{KerrLi2}]\label{Paunescu1}
If $G$ is an initially subamenable group, then there exists a sofic embedding $\pi: G\to \cS$ such that $\pi(G)'\cap \cS$ acts ergodically on $L^{\infty}(\mathcal{L})$.
\end{lemm}

\begin{lemm}[Theorem 2.10 in \cite{Pau1} combined with Theorem 2 of \cite{elekxzabo}]\label{Paunescu2}
If $G$ is a non amenable sofic group, then there exists a sofic embedding $\pi: G\to \cS$ such that $\pi(G)'\cap \cS$  does not act ergodically on $L^{\infty}(\mathcal{L})$.
\end{lemm}
\begin{proof}
    Using two non conjugate sofic embeddings from Theorem 2 of \cite{elekxzabo}, one takes a non trivial convex combination (in the sense of \cite{Pau1}) to obtain an embedding $\pi: G\to \cS$ such that $\pi(G)'\cap \cS$  does not act ergodically on $L^{\infty}(\mathcal{L})$ by Theorem 2.10 of \cite{Pau1}.
\end{proof}

\begin{defn}
Let $G$ be a sofic group. Say that it satisfies the \emph{generalized Elek-Szabo property} if for any two embeddings $\pi_1,\pi_2: G\to \mathcal{S}$, there exists $\Phi\in Aut(\mathcal{S})$ satisfying $\Phi\circ \pi_1=\pi_2$.
\end{defn}

We are now ready to prove the main result.

\begin{thm}\label{any two embeddings are conjugate by automorphism implies amenable}
Let $G$ be a sofic group that is initially subamenable. Then $G$ satisfies the generalized Elek-Szabo property if and only if $G$ is amenable.

\end{thm}

\begin{proof}
If $G$ is amenable, this follows from Theorem 2 in \cite{elekxzabo}. Conversely, if $G$ is initially subamenable but not amenable, from Lemma \ref{Paunescu1} there exists a sofic embedding $\pi: G\to \cS$ such that $\pi(G)'\cap \cS$ acts ergodically on $L^{\infty}(\mathcal{L})$. Moreover, from Lemma \ref{Paunescu2} there exists a sofic embedding $\rho: G\to \cS$ such that $\rho(G)'\cap \cS$  does not act ergodically on $L^{\infty}(\mathcal{L})$. These two embeddings cannot be automorphically conjugate from Lemma \ref{main lemma Jung}.
\end{proof}

\section{Automorphic conjugacy in the context of multiples}\label{mul}

In this section, we prove a generalization of Theorem \ref{any two embeddings are conjugate by automorphism implies amenable} where we are allowed to take multiples of our sofic embeddings. 
We thank the referee for the suggestion to consider and include results in this  general setting. 

\begin{defn}
Let $G$ be a group and let $\sigma_{j}\colon G\to \Sym_{d_{j}}$  for $j=1,2$ be maps. Define their \emph{direct sum}
\[\sigma_{1}\oplus \sigma_{2}\colon G\to \Sym_{d_{1}+d_{2}}\]
by
\[(\sigma_{1}\oplus \sigma_{2})(g)(j)=\begin{cases}
    \sigma_{1}(g)(j), &\textnormal{ if $1\leq j\leq d_{1}$}\\
    \sigma_{2}(g)(j-d_{1})+d_{1}, &\textnormal{ if $d_{1}+1\leq j\leq d_{1}+d_{2}$.}
\end{cases}\]
We use $\sigma^{\oplus r}$ for the direct sum of $\sigma$ with itself $r$-times.
\end{defn}
Note that if $(\sigma_{n,j})_{n},j=1,2$ are sofic approximations, then so is $\sigma_{n,1}\oplus \sigma_{n,2}$.
We recall the equivalence between sofic approximations introduced by P\v{a}unescu. Let $\sigma_{n,j}\colon G\to \Sym_{k_{n,j}}$ be sofic approximations and $\omega$ a free ultrafilter on $\bN$. In \cite{Pau1}, P\v{a}unescu defined  that $\sigma_{n,1}$ and $\sigma_{n,2}$ to be equivalent if there are nonnegative integers $q_{n,j},j=1,2$ so that $k_{n,1}q_{n,1}=k_{n,2}q_{n,2}$ and a $\pi\in \prod_{n\to\omega}(\Sym_{k_{n,1}q_{n,1}},d_{k_{n,1}q_{n,1}})$ with 
\[\pi(\sigma_{n,1}^{\oplus q_{n,1}}(g))_{n\to\omega}\pi^{-1}=(\sigma_{n,2}^{\oplus q_{n,2}}(g))_{n\to\omega}, \textnormal{ for all $g\in G$}.\]
In \cite{Pau1} it was also shown how to embed the space of equivalence classes as a closed convex subset of a Banach space. We will not go into the precise definition of the convex combination of two sofic approximations and refer the reader to \cite{Pau1} for the definitions. For a sofic approximation $\sigma_{n} \colon G\to \Sym_{k_{n}}$, we use $[(\sigma_{n})_{n}]$ for the equivalence class of $(\sigma_{n})_{n}$ under this equivalence. We already implicitly used this convex structure in the proof of Theorem \ref{any two embeddings are conjugate by automorphism implies amenable}. Since we need to use this structure in a more explicit manner in the context of multiples, we highlight the main features we used about it in the proof of Theorem \ref{any two embeddings are conjugate by automorphism implies amenable} (which we will once again need in this more general setting of multiples):
\begin{itemize}
    \item if $G$ is a sofic group, then the above space of equivalence classes reduces to a single point if and only if $G$ is amenable (this is a consequence of the uniqueness theorems of Kerr-Li, Elek-Szabo, see \cite[Observation 1.9]{Pau1}),
    \item as a consequence of the above, if $G$ is a nonamenable sofic group, then there is a sofic approximation which is not extremal in P\v{a}unescu's convex structure, (take a nontrivial convex combination of two inequivalent sofic approximations),
    \item if $\sigma$ is a sofic approximation with ergodic commutant, then $[\sigma]$ is extremal \cite[Theorem 2.10]{Pau1}.
\end{itemize}

For $k\in \bN$, we use $t_{k}\colon G\to \Sym_{k}$ for the trivial homomorphism.

\begin{lemm}\label{lemm: I guess this has to be written}

Let $\sigma_{n}\colon G\to \Sym_{k_{n}}$ be a sofic approximation. Let $r_{n}$ be a sequence of integers such that $\frac{r_{n}}{k_{n}+r_{n}}\to_{n\to\infty}0$. Let $\widetilde{\sigma}_{n}=\sigma_{n}\oplus t_{r_{n}}$. Fix a free ultrafiler $\omega$ on $\bN$. Then $(\sigma_{n})_{n}$ has ergodic commutant with respect to $\omega$ if and only if $(\widetilde{\sigma}_{n})_{n}$ has ergodic commutant with respect to $\omega$.

\end{lemm}

\begin{proof}
Set 
\[\cS=\prod_{n\to\omega}(\Sym_{k_{n}},d_{k_{n}}),\]
\[\widetilde{\cS}=\prod_{n\to\omega}(\Sym_{k_{n}+r_{n}},d_{k_{n}+r_{n}}).\]
Let $\sigma\colon G\to \cS$, $\widetilde{\sigma}\colon G\to \widetilde{\cS}$ be the homomorphisms induced from
$(\sigma_{n})_{n}$, $(\widetilde{\sigma}_{n})_{n}$.
If $(\psi_{n})_{n\to\omega}\in \sigma(G)'\cap \cS$, then $(\psi_{n}\oplus t_{r_{n}})_{n\to\omega}\in \widetilde{\sigma}(G)'\cap \cS$. Conversely, if $(\phi_{n})_{n\to\omega}\in \widetilde{\sigma}(G)'\cap \widetilde{\cS}$, set $E_{n}=\phi_{n}^{-1}(\{1,\cdots,k_{n}\})\cap \{1,\cdots,k_{n}\}$. Since $\frac{r_{n}}{k_{n}+r_{n}}\to_{n\to\infty}0$, we have that $\frac{|E_{n}|}{k_{n}}\to 1.$ For each $n$, choose an arbitrary bijection
\[\alpha_{n}\colon \{1,\cdots,k_{n}\}\setminus E_{n}\to \{1,\cdots,k_{n}\}\setminus [\phi_{n}(\{1,\cdots,k_{n}\})\cap \{1,\cdots,k_{n}\}],\]
and set $\widehat{\phi}_{n}=\phi_{n}\big|_{E_{n}}\sqcup \alpha_{n}$. It is then direct to check from the fact that $(\phi_{n})_{n\to\omega}\in \widetilde{\sigma}(G)'\cap \widetilde{\cS}$ that $(\widehat{\phi}_{n})_{n\to\omega}\in \sigma(G)'\cap \cS.$ Moreover these operations are inverse to each other (remember in $\cS$ we mod out by sequences whose distance to each other tends to $0$ along $\omega$).

From these observations, the lemma can be proved as follows. Set
\[L^{\infty}(\cL)=\prod_{n\to\omega}
D_{k_{n}}(\bC),\]
\[L^{\infty}(\widetilde{\cL})=\prod_{n\to\omega}D_{k_{n}+r{n}}(\bC).\]
If $f=(f_{n})_{n\to\omega}\in L^{\infty}(\cL)$ is fixed by $\sigma(G)'\cap \cS$ and not a scalar, then $(f_{n}1_{\{1,\cdots,k_{n}}\})_{n\to\omega}$ is fixed by $\widetilde{\sigma}(G)'\cap \widetilde{\cS}$ and not a scalar. Conversely, if $\widetilde{f}=(\widetilde{f}_{n})_{n\to\omega}\in L^{\infty}(\widetilde{\cL})$ is fixed by $\widetilde{\sigma}(G)'\cap \widetilde{\cS}$ and not a scalar, then $(\widetilde{f}_{n}|_{\{1,\cdots,k_{n}\}})_{n\to\omega}$ is fixed by $\sigma(G)'\cap \cS$ and not a scalar.

\end{proof}

\begin{prop}\label{prop: multiples, I guess}
Let $G$ be a countable group and suppose that 
$\sigma_{n,j}\colon G\to \Sym(m_{n,j})$
are sofic approximations for $j=1,2$.
Fix $\omega\in \beta(\bN) \setminus \bN$ and 
for $j=1,2$ define sofic embeddings $\sigma_{j}$  of $G$ by $\sigma_{j}(g)=(\sigma_{n,j}(g))_{n\to\omega}$.
Suppose that $\sigma_{1}$ has ergodic commutant and that  $[\sigma_{2}]$ is not extremal in P\v{a}unescu's convex structure.  

Then $\sigma_{1},\sigma_{2}$ do not have automorphically conjugate multiples in the following sense. Suppose we have $(m_{n,j})_{n=1}^{\infty}$, $(q_{n,j})_{n=1}^{\infty},(r_{n,j})_{n=1}^{\infty}$ sequences of integers for $j=1,2$ satisfying:
\[q_{n,1}m_{n,1}+r_{n,1}=q_{n,2}m_{n,2}+r_{n,2} \textnormal{ for all $n$, and}\]
\[\lim_{n\to\infty}\frac{r_{n,j}}{q_{n,j}m_{n,j}+r_{n,j}}=0, \textnormal{ for $j=1,2$}.\]
For $j=1,2$ define sofic approximations
\[\phi_{n,j}\colon G\to \Sym(k_{n})\]
by $\phi_{n,j}=\sigma_{n,j}^{\oplus q_{n,j}}\oplus t_{r_{n,j}}$. For $j=1,2$ set $\phi_{j}(g)=(\phi_{n,j}(g))_{n\to\omega}.$
Then $\phi_{1},\phi_{2}$ are not automorphically conjugate.

\end{prop}

\begin{proof}
By Lemma \ref{main lemma Jung}, it suffices to show $\phi_{1}$ has ergodic commutant and that $\phi_{2}$ does not. 
It follows from \cite[Proposition 2.8]{Pau1} and Lemma \ref{lemm: I guess this has to be written} that $\phi_{1}$ has ergodic commutant. Set 
\[\cS_{2}'=\prod_{k\to\omega}(\Sym_{q_{n,2}m_{n,2}}).\]
and define $\phi_{2}'\colon G\to \cS_{2}'$ by $\phi_{2}'(g)=(\sigma_{n,2}^{\oplus q_{n,2}}(g))_{n\to\omega}$. Then $\phi_{2}'$ and $\sigma_{2}$ are equivalent in P\v{a}unescu's convex structure, so $\phi_{2}'$ is also not extremal in this convex structure. By \cite[Theorem 2.10]{Pau1}, it follows that $\phi_{2}'$ does not have ergodic commutant. In particular, by Lemma \ref{lemm: I guess this has to be written} we have that $\phi_{2}$ does not have ergodic commutant.

\end{proof}

\begin{cor}\label{cor: multiples, I guess}
Let $G$ be an nonamenable group which is initially subamenable. Fix a free ultrafilter $\omega\in \beta\bN\setminus \bN$. Then there are sofic approximations $\sigma_{n,j}\colon G\to \Sym_{m_{n,j}},j=1,2$
so that if $q_{n,j},r_{n,j}$ are any sequences of integers with $q_{n,1}m_{n,1}+r_{n,1}=q_{n,2}m_{n,2}+r_{n,2}$ and 
\[\lim_{n\to\infty}\frac{r_{n,j}}{q_{n,j}k_{n,j}+r_{n,j}}=0, \textnormal{ for $j=1,2$}\]
then setting $\phi_{j}=(\sigma_{n,j}^{\oplus q_{n,j}}\oplus t_{r_{n,j}})_{n\to\omega}$ we have that $\phi_{1},\phi_{2}$ are not automorphically conjugate.

\end{cor}

\begin{proof}
 Since $G$ is intially subamenable, from Lemma \ref{Paunescu1} we can choose a sofic approixmation $\sigma_{1}=(\sigma_{n,1}\colon G\to \Sym_{m_{n,1}})_{n}$ as in the statement of the corollary so that that $g\mapsto (\sigma_{n,1}(g))_{n\to\omega}$ has ergodic commutant.  Since $G$ is not amenable, by taking nontrivial convex combinations  we can choose $\sigma_{2}=(\sigma_{n,2}\colon G\to \Sym_{m_{n,2}})_{n}$ as in the statement of the corollary so that $g\mapsto (\sigma_{n,2}(g))_{n\to\omega}$ is not extremal in P\v{a}unescu's convex structure. The corollary now follows from Proposition \ref{prop: multiples, I guess}.
\end{proof}

\textbf{Acknowledgments.}  We especially thank  Isaac Goldbring,  Jenny Pi and the anonymous referee for their several helpful comments that significantly helped improve the writing. We also thank Francesco Fournier-Facio and Liviu Pa\u{u}nescu for helpful comments.
B. Hayes gratefully acknowledges support from the NSF grant DMS-2144739. S. Kunnawalkam Elayavalli gratefully acknowledges support from the Simons Postdoctoral Fellowship.

%

\end{document}